\documentclass[letterpaper,11pt]{amsart}

\usepackage[margin=1.2in]{geometry}
\usepackage{amsmath,amsthm,amssymb, mathtools}
\usepackage{xspace,xcolor}
\usepackage[breaklinks,colorlinks,citecolor=teal,linkcolor=teal,urlcolor=teal,pagebackref,hyperindex]{hyperref}
\usepackage[alphabetic]{amsrefs}
\usepackage{comment}
\usepackage{enumerate}
\usepackage[all]{xy}
\usepackage{tikz-cd}
\usepackage{color}
\usepackage{comment}
\usepackage{kotex}

\theoremstyle{plain}
\newtheorem{theo}{Theorem}[section]

\newtheorem{prop}[theo]{Proposition}
\newtheorem{coro}[theo]{Corollary}

\newtheorem{ques}{Question}

\theoremstyle{definition}
\newtheorem{defi}[theo]{Definition}

\theoremstyle{remark}
\newtheorem{rema}[theo]{Remark}

\newcommand{\eps}{\epsilon}

\newcommand{\bfR}{\mathbf{R}}

\newcommand{\HH}{\mathbb{H}}

\newcommand{\QQ}{\mathbb{Q}}

\newcommand{\CC}{\mathbf{C}}

\newcommand{\ZZ}{\mathbb{Z}}

\newcommand{\cC}{\mathcal{C}}
\newcommand{\cD}{\mathcal{D}}

\newcommand{\cH}{\mathcal{H}}

\newcommand{\cO}{\mathcal{O}}

\newcommand{\lsta}{_{*}}

\DeclareMathOperator{\codim}{codim}

\DeclareMathOperator{\gr}{gr}

\newcommand{\SheafHom}{\mathcal{H}om}

\newcommand{\SheafExt}{\mathcal{E}xt}

\newcommand{\duBois}{\underline{\Omega}}

\DeclareMathOperator{\sing}{sing}

\DeclareMathOperator{\lcd}{lcd}
\DeclareMathOperator{\lcdef}{lcdef}

\newcommand{\Gamunder}{\underline{\Gamma}}

\usepackage{soul}

\begin{document}
	\thanks{The author was partially supported by NSF grant DMS-2301463 and the Simons Collaboration grant Moduli of Varieties.}
	
	
	\author{Hyunsuk Kim}
	
	\address{Department of Mathematics, University of Michigan, 530 Church Street, Ann Arbor, MI 48109, USA}
	
	\email{kimhysuk@umich.edu}	


	
	
	\begin{abstract} 
         We construct some version of the trace morphism between the Du Bois complexes, with applications towards the behavior of the local cohomological dimension and some Hodge theoretic aspects of singularities under finite morphisms.
	\end{abstract} 
	
	\title[Trace for the Du Bois complex]{Trace for the Du Bois complex}
	
	\maketitle
	
 
    \section{Introduction}
    Let $X$ be a variety, i.e., a reduced, separated scheme of finite type which is defined over $\CC$ (non-necessarily irreducible). In \cite{DuBois:complexe-de-deRham}, Du Bois constructs a filtered complex $(\duBois_{X}^{\bullet}, F^{\bullet})$ which is the correct replacement of the de Rham complex when $X$ is singular. In fact, it is the sheaf theoretic object that computes the Hodge structure of an arbitrary proper variety constructed by Deligne. The $p$-th graded piece of this complex (with a suitable shift) $\duBois_{X}^{p}$ lives in the bounded derived category of coherent sheaves, and it plays an important role in the recent development of `higher' notions of Du Bois and rational singularities, for example, \cite{MP-higher-rational}, \cite{FL-isosing}, \cite{FL-HDuBois}, \cite{Jung-Kim-Saito-Yoon}, \cite{MOPW}, \cite{CDM}, \cite{SVV-DB}, \cite{PSV}, \cite{park-popa}, \cite{kovacs-injectivity} for a non-exhaustive list. The question that we want to answer in this article is the following:
    \begin{ques}
        Let $\pi \colon Y \to X$ be a finite surjective morphism between normal algebraic varieties. Is there a version of the trace map for the Du Bois complex?
    \end{ques}
    Toward this question, we prove the following result.
 
    \begin{theo} \label{theo:main}
        Let $X$ be a variety and $\pi = \pi_{X} \colon X \to X'$ be a finite group quotient by $G$. Then the composition of the natural morphisms
        $$ \duBois_{X'}^{\bullet} \to \bfR\pi\lsta \duBois_{X}^{\bullet} \to \bfR \Gamunder^{G} \bfR\pi\lsta \duBois_{X}^{\bullet} $$
        is an isomorphism in the filtered derived category. In particular, after taking graded pieces, the following compositions are isomorphisms for all $p$:
        $$ \duBois_{X'}^{p} \to \bfR\pi\lsta \duBois_{X}^{p} \to \bfR \Gamunder^{G} \bfR \pi\lsta \duBois_{X}^{p}.$$
    \end{theo}
    We remark that both $\pi\lsta$ and $\Gamunder^{G}$ are exact functors. This morphism is the one in \cite{DuBois:complexe-de-deRham}*{(5.11.2)}, and we briefly explain its construction in \S\ref{section:DuBois}. 

    \begin{rema}
        Du Bois shows this statement when $X$ is smooth (see \cite{DuBois:complexe-de-deRham}*{Théorème 5.12}). In fact, the proof of Theorem \ref{theo:main} reduces to this situation.
    \end{rema}
    
    We collect some immediate corollaries.
    \begin{coro} \label{coro:trace}
        Let $f \colon Y \to X$ be a morphism between algebraic varieties over $\CC$. If (1) $f$ is a finite group quotient, or (2) $f$ is a finite surjective morphism and $X$ is normal, then there is a morphism $t \colon \bfR f\lsta \duBois_{Y}^{p} \to \duBois_{X}^{p}$ such that the composition
        $$ \duBois_{X}^{p} \to \bfR f\lsta \duBois_{Y}^{p} \xrightarrow{t} \duBois_{X}^{p} $$
        is an isomorphism.
    \end{coro}
    
    \begin{coro} \label{coro:improve-under-finite}
        Let $f \colon Y \to X$ be a morphism between algebraic varieties over $\CC$. If (1) $f$ is a finite group quotient, or (2) $f$ is a finite surjective morphism and $X$ is normal, then
        \begin{enumerate}
            \item $\lcdef(X) \leq \lcdef(Y)$.
            \item If $Y$ has pre-$m$-Du Bois singularities, then so is $X$.
            \item If $Y$ has pre-$m$-rational singularities,  then so is $X$.
            \item Furthemore, if $Y$ and $X$ are proper, we have the inequality between the Hodge--Du Bois numbers $\underline{h}^{p,q}(Y) \geq \underline{h}^{p,q}(X)$.
        \end{enumerate}
        In particular, $\lcdef(Y) = 0$ implies that $\lcdef(X) = 0$.
    \end{coro}

    The Hodge--Du Bois number of a variety $X$ is defined as $\underline{h}^{p,q}(X) \coloneqq \dim \gr_{F}^{p} H^{p+q}(X)$, where $F$ is the Hodge filtration on the mixed Hodge structure on the singular cohomology of $X$.

    \begin{rema}
        We point out that in the case of a finite group quotient, we don't assume the normality of $X$.
    \end{rema}

    \begin{rema}
        We remark that (2) is already proved in \cite{SVV-DB}*{Proposition 4.2} and (3) is proved in \cite{DOR}*{Corollary 6.6} when $f$ is a finite group quotient. After finishing this article we learned that Kovács--Lank--Venkatesh obtained the same result as in (3) using the recent injectivity theorem of Kovács \cite{kovacs-lank-venkatesh}, and that Duc Vo also obtained similar results in his thesis \cite{Duc}. We hope that the existence of the trace morphism gives a clearer explanation for these statements.
    \end{rema}

    \section{Preliminaries}
    \subsection{The Du Bois complex}  \label{section:DuBois}
    We briefly explain the Du Bois complex $\duBois_{X}^{\bullet}$ of a variety, constructed in \cite{DuBois:complexe-de-deRham}. We refer to \cite{DuBois:complexe-de-deRham}*{\S1} for the details of the construction of the Du Bois complex, and \cite{DuBois:complexe-de-deRham}*{\S5} for the equivariant version. One can construct a simplicial resolution $\eps_{\bullet} \colon X_{\bullet} \to X$ from a cubical resolution of $X$ (for example, \cite{Guillen-Navarro-Gainza:Hyperresolutions-cubiques}). Then the Du Bois complex is defined as
    $$ (\duBois_{X}^{\bullet}, F^{\bullet}) := \bfR \eps_{\bullet \ast} (\Omega_{X_{\bullet}}^{\bullet}, F^{\bullet}),$$
    where the filtration on the right hand side is given by the `filtration bête'. This object lives in the 
     filtered derived category of differential complexes of order $\leq 1$, which we denote by $\cD_{\mathrm{diff}}(X)$. The graded pieces $\duBois_X^{p} := \gr_{F}^{p} \duBois_{X}^{\bullet}[p]$ live in the derived category of coherent sheaves on $X$. If one has a morphism $f \colon Y \to X$, then the functor $f\lsta$ defined on the level of complexes induces $\bfR f\lsta \colon \cD_{\mathrm{diff}}(Y) \to \cD_{\mathrm{diff}}(X)$. Taking the push-forward commutes with taking the graded pieces. i.e., for $K^{\bullet}$ in $\cD_{\mathrm{diff}}(Y)$, we have the canonical isomorphism
    $$ \bfR f\lsta \gr_{F}^{p} K^{\bullet} \simeq \gr_{F}^{p} \bfR f\lsta K^{\bullet},$$
    where the $\bfR  f\lsta$ on the left is the usual push-forward between the derived categories of $\cO$-modules.
    
    Let $G$ be a finite group. If a variety $Y$ has a $G$-action, one can define $\duBois_{Y}^{\bullet}$ in a $G$-equivariant manner using $G$-equivariant resolutions. We note that $\duBois_{Y}^{\bullet}$ lives in the \textit{derived category of equivariant filtered differential complexes}, $\cD_{\mathrm{diff}}(Y, G)$. We denote by $\cC_{\mathrm{diff}}(Y, G)$ the category of equivariant differential complexes. If $G$ acts trivially on $Y$, then one can take the invariant part $\Gamunder^{G} K^{\bullet}$. The functor $\Gamunder^{G} \colon \cC_{\mathrm{diff}}(Y, G) \to \cC_{\mathrm{diff}}(Y)$ derives itself to
    $$ \bfR \Gamunder^{G} \colon \cD_{\mathrm{diff}}(Y, G) \to \cD_{\mathrm{diff}}(Y).$$
    Note that since $G$ is finite, there is a canonical natural transformation from the forgetful functor to $\bfR \Gamunder^{G}$ since we are in characteristic zero.
    
    Therefore, if we consider a $G$-equivariant morphism $f \colon Y \to X$, where $G$ acts trivially on $X$, we get a natural morphism
    $$ \duBois_{X}^{\bullet} \to \bfR f\lsta \duBois_{Y}^{\bullet} \to \bfR \Gamunder^{G} \bfR f\lsta \duBois_{Y}^{\bullet}.$$
    We refer to \cite{DuBois:complexe-de-deRham}*{\S5} for details. This is the morphism that we study in Theorem \ref{theo:main}.

    \subsection{Local cohomological dimension} Consider a singular variety $X$ embedded in a smooth variety $Y$. Then, we can consider the local cohomology sheaves $\cH_{X}^{q} \cO_{Y}$. The local cohomology starts at the codimension, i.e.,
    \[ \codim_{Y} X = \min \{ q : \cH_{X}^{q} \cO_{Y} \neq 0\}.\]
    The other end of the cohomological amplitude is by definition the local cohomological dimension, i.e.,
   	\[ \lcd_{Y}X := \max \{ q : \cH_{X}^{q} \cO_{Y} \neq 0\}.\]
   The difference between these two objects $\lcd_{Y}X - \codim_{Y}X$ is called the local cohomological defect, which we denote by $\lcdef(X)$, following the notation of \cite{Popa-Shen:DuBoisLCDEF}. This quantity does not depend on the choice of the embedding $X \hookrightarrow Y$. By studying the Hodge module structure of local cohomology, the local cohomological defect can be completely characterized by the vanishing and non-vanishing behavior of the Grothendieck dual of the Du Bois complex.
   \begin{prop}[\cite{Mustata-Popa22:Hodge-filtration-local-cohomology}*{Corollary 5.3}] \label{theo:MP-lcd-intermsof-DuBois}
   	The local cohomological defect $\lcdef(X)$ is the largest integer $c$ satisfying the following:
   	\begin{enumerate}
   		\item $\SheafExt_{\cO_{X}}^{j+c}(\duBois_{X}^{j}, \omega_{X}^{\bullet}[-\dim X]) \neq 0$ for some $j \geq 0$.
   		\item $\SheafExt_{\cO_{X}}^{j+c+1} (\duBois_{X}^{j}, \omega_{X}^{\bullet}[-\dim X]) = 0$ for all $j \geq 0$.
   	\end{enumerate}
   \end{prop}
	We point out that we are restating \cite{Mustata-Popa22:Hodge-filtration-local-cohomology}*{Corollary 5.3} by applying Grothendieck duality to the inclusion $\iota \colon X \hookrightarrow Y$.

    We also note that the local cohomological defect is topological in nature. Using the Riemann--Hilbert correspondence and (Verdier) duality, we have
    \begin{theo}
        $ \lcdef(X) = \max \{ j : \prescript{p}{}{\cH}^{-j} (\QQ_{X}[\dim X]) \neq 0\}.$
    \end{theo}
    Here, $\prescript{p}{}{\cH}$ is the perverse cohomology. See for example \cite{RSW-lcdef-intermsofTopology} for details.
	
	\subsection{Higher Du Bois and Higher rationality} In \cite{SVV-DB}, the authors generalize the notion of higher Du Bois and rational singularities only involving the behavior of the Du Bois complex, in an attempt to find the natural definition outside of the local complete intersection setting. In that vein, they define
	\begin{defi}
		Let $X$ be a variety of dimension $n$. Then $X$ has pre-$m$-Du Bois singularities if
		\[ \cH^{i} (\duBois_{X}^{p}) = 0 \qquad \text{for all } 0\leq p \leq m,\text{ and } i > 0.\]
		$X$ has pre-$m$-rational singularities if
		\[ \SheafExt_{\cO_{X}}^{i}(\duBois_{X}^{n-p}, \omega_{X}^{\bullet}[-\dim X]) = 0 \qquad \text{ for all } 0 \leq p \leq m \text{ and } i > 0.\]
	\end{defi}
    Note that with some additional conditions (bounds on the codimension of the singular locus and the reflexivity of $\cH^{0}(\duBois_{X}^{p})$), these conditions are equivalent to the original definition in the local complete intersection case (see \cite{Jung-Kim-Saito-Yoon} or  \cite{FL-HDuBois}).

    \section{Proof of the theorem}
    Here, we give the proof of Theorem \ref{theo:main}. We prove it by induction on the dimension of $X$. If $\dim X = 0$, there is nothing to show. Since $G$ acts on $X$, the singular locus $Z = X_{\sing}$ immediately carries a $G$-action. Consider a $G$-equivariant projective resolution of singularities $\mu \colon \widetilde{X} \to X$, which is an isomorphism over the smooth locus of $X$. Let $E = \mu^{-1}(Z)_{\mathrm{red}}$. Note that $\widetilde{X}, X, E,$ and $Z$ have a $G$-action. By assumption, $X' = X/G$ and we denote by $\widetilde{X}', E',$ and $Z'$ analogously for the quotient spaces. Note that these quotients exist as varieties since $\mu$ is a projective morphism. Similar to the quotient morphism $\pi_{X} \colon X \to X'$, we denote the quotient morphisms by $\pi_{\widetilde{X}}, \pi_{E}, \pi_{Z}$ analogously. Then we have two commutative diagrams
    $$ \begin{tikzcd}
        E \ar[r] \ar[d, "\mu_{E}"] & \widetilde{X} \ar[d, "\mu"] \\ Z \ar[r] & X
    \end{tikzcd} \qquad \text{and} \qquad  \begin{tikzcd}
        E' \ar[r] \ar[d, "\mu_{E}'"] & \widetilde{X}' \ar[d, "\mu'"] \\ Z' \ar[r] & X'.
    \end{tikzcd} $$
    Note that $\mu$ and $\mu'$ are isomorphisms over $X \setminus Z$ and $X' \setminus Z'$ respectively. Then we have the following commutative diagram:
    $$
    \begin{tikzcd}
        \duBois_{X'}^{\bullet} \ar[r] \ar[d] & \bfR \mu\lsta'\duBois_{\widetilde{X}}^{\bullet} \oplus \duBois_{Z'}^{\bullet} \ar[r] \ar[d] & \bfR \mu_{E,\ast}' \duBois_{E'}^{\bullet} \xrightarrow{+1}\ar[d] \\
        \bfR \pi_{X,\ast} \duBois_{X}^{\bullet} \ar[r] \ar[d] & \bfR \mu\lsta' \bfR \pi_{\widetilde{X}, \ast} \duBois_{\widetilde{X}}^{\bullet} \oplus \bfR \pi_{Z,\ast} \duBois_{Z}^{\bullet} \ar[r] \ar[d]& \bfR \mu_{E,\ast}' \bfR \pi_{E,\ast} \duBois_{E}^{\bullet} \xrightarrow{+1} \ar[d]\\
        \bfR \Gamunder^{G}\bfR \pi_{X,\ast} \duBois_{X}^{\bullet} \ar[r] & \bfR \Gamunder^{G}\bfR \mu\lsta' \bfR \pi_{\widetilde{X}, \ast} \duBois_{\widetilde{X}}^{\bullet} \oplus \bfR \Gamunder^{G}\bfR \pi_{Z,\ast} \duBois_{Z}^{\bullet} \ar[r]& \bfR \Gamunder^{G}\bfR \mu_{E,\ast}' \bfR \pi_{E,\ast} \duBois_{E}^{\bullet} \xrightarrow{+1}. 
    \end{tikzcd}
    $$
    The horizontal rows are exact triangles. Note that $\mu'$ and $\mu_{E}'$ are morphisms between varieties with trivial $G$-action. Therefore, $\bfR \Gamunder^{G}$ commutes with $\bfR \mu\lsta'$ and $\bfR \mu_{E,\ast}'$. We assert that the composition of the vertical arrows give isomorphisms. Since $\dim Z , \dim E < \dim X$ we have by induction that
    $$ \duBois_{Z'}^{\bullet} \xrightarrow{\simeq} \bfR \Gamunder^{G} \bfR \pi_{Z, \ast} \duBois_{Z}^{\bullet}$$
    and similar for $E$. Also, since $\widetilde{X}$ is smooth, we have the analogous isomorphism for $\widetilde{X}$ by \cite{DuBois:complexe-de-deRham}*{Théorème 5.12}. Since the compositions of the second and the third vertical arrows are isomorphisms, the composition of the first vertical arrows
    $$ \duBois_{X'}^{\bullet} \to \bfR \Gamunder^{G} \bfR \pi_{X, \ast} \duBois_{X}^{\bullet}$$
    is an isomorphism in the filtered derived category. The second assertion follows from the first one by taking graded pieces. \hfill{$\square$}

    \section{Proof of the corollaries}
    \begin{proof}[Proof of Corollary \ref{coro:trace}]
        If $X = Y/G$ for $G$ a finite group acting on $X$, this is exactly the second statement of Theorem \ref{theo:main}. For the second case, consider $\gamma \colon Z \to Y^{n} \xrightarrow{\nu} Y \to X$ such that $Z$ is Galois over $X$ and $Y^{n} \to Y$ is the normalization. Then, the composition of the morphisms
        $ \duBois_{X}^{p} \to \bfR f\lsta \duBois_{Y}^{p}  \to \bfR (f\circ \nu)\lsta \duBois_{Y^{n}}^{p} \to \bfR \gamma\lsta \duBois_{Z}^{p} \to \duBois_{X}^{p}$
        is an isomorphism.
    \end{proof}

    \begin{proof}[Proof of Corollary \ref{coro:improve-under-finite}]
        By applying Grothendieck duality for the map $f$ to Corollary \ref{coro:trace}, we see that the composition
        $$ \bfR\SheafHom_{\cO_{X}}(\duBois_{X}^{n-p},\omega_{X}^{\bullet}) \to \bfR f\lsta \bfR \SheafHom_{\cO_{Y}}(\duBois_{Y}^{n-p}, \omega_{Y}^{\bullet}) \to \bfR\SheafHom_{\cO_{X}}(\duBois_{X}^{n-p},\omega_{X}^{\bullet})$$
        is an isomorphism. Since $f$ is a finite morphism, $f\lsta$ is exact. Therefore,
        $$ \SheafExt_{\cO_{Y}}^{q} (\duBois_{Y}^{n-p}, \omega_{Y}^{\bullet}) = 0 \implies \SheafExt_{\cO_{X}}^{q} (\duBois_{X}^{n-p}, \omega_{X}^{\bullet}) = 0.$$
        This immediately shows the third statement. The first assertion also follows from the characterization of the local cohomological defect in terms of the depth of the Du Bois complex (see Proposition \ref{theo:MP-lcd-intermsof-DuBois}). The second assertion is also immediate from Theorem \ref{theo:main}. The last assertion follows from the fact that the Hodge--Du Bois number of a proper variety $X$ can be computed by the cohomology of the Du Bois complex, i.e., $\underline{h}^{p,q}(X) = \dim \HH^{q}(X, \duBois_{X}^{p})$.
    \end{proof}

    \begin{rema}
        We also immediately recover the other statement of \cite{DOR}*{Corollary 6.6} as well. Let $\pi \colon Y \to X$ be a finite group quotient and let $\dim X = \dim Y = n$. Then we have a commutative diagram
        $$ \begin{tikzcd}
            \duBois_{X}^{p} \ar[r] \ar[d] & \bfR \SheafHom_{\cO_{X}} (\duBois_{X}^{n-p}, \omega_{X}^{\bullet}[-n])  \\
            \pi\lsta\duBois_{Y}^{p} \ar[r] & \pi\lsta \bfR \SheafHom_{\cO_{Y}} (\duBois_{Y}^{n-p}, \omega_{Y}^{\bullet}[-n]) \ar[u].
        \end{tikzcd}$$
        The statement of \cite{DOR}*{Corollary 6.6} says that if the bottom row is an isomorphism for $p \leq m$, then the top row is as well for $p \leq m$. Note that after taking cohomology, the first vertical map is an inclusion onto the $G$-invariant part and the second map sends the $G$-invariant part isomorphically to the target. Therefore, if the bottom row is an isomorphism, then the top row is an isomorphism in cohomology (hence an isomorphism in the derived category). Note that $\pi\lsta$ is exact.
    \end{rema}
    
    \begin{rema}
    	We note that the local cohomological defect can be strictly smaller after taking a finite quotient. For example, consider an abelian surface $A$ with a $G = \ZZ/2\ZZ$-action given by the negation. If we consider a $G$-equivariant very ample line bundle and consider the cone $X$ over $A$ associated to that line bundle, then $\lcdef(X) = 1$, since $\prescript{p}{}{\cH}^{-1}(\QQ_{X}[3]) \simeq i_{x\ast} H^{1}(A, \QQ)$, where $i_{x} : \{x \} \to X$ is the closed immersion from the cone point. However, $X' = X/G$ is a cone over $A/G$ and $H^{1}(A/G, \QQ) = 0$ since $A/ G$ is simply connected (it is klt and its minimal resolution is a $K3$ surface). This implies $\lcdef(X') = 0$ by the same reason since $A/G$ has quotient singularities.
    \end{rema}

    \textbf{Acknowledgements.} The author would like to thank Mircea Musta\c{t}\u{a} and Sridhar Venkatesh for many discussions. The author also thanks Sándor Kovács for pointing out that stronger versions of the corollaries can be formulated.

    \bibliographystyle{alpha}
    \bibliography{Reference}

\end{document}